\documentclass[reqno,b5paper]{amsart}
\usepackage{amsmath}
\usepackage{amssymb}
\usepackage{amsthm,amsxtra}
\usepackage{cite}
\usepackage{enumerate}
\usepackage[mathscr]{eucal}

\newtheorem{Def}{Definition}[section]
\newtheorem{Th}{Theorem}[section]
\newtheorem{Ex}{Example}[section]

\newtheorem{Lemma}{Lemma}[section]
\newtheorem{Prop}{Proposition}[section]
\newtheorem{Cor}{Corollary}[section]

{}

\newcommand{\br}{\mathbb{R}}
\newcommand{\bn}{\mathbb{N}}

\newcommand{\jc}{\mathcal{J}}
\newcommand{\ic}{\mathcal{I}}
\newcommand{\ac}{\mathcal{A}}

\newcommand{\fc}{\mathcal{F}}

\newcommand{\pc}{\mathcal{P}}
\newcommand{\cs}{\mathcal{S}}
\newcommand{\cu}{\mathcal{U}}

\begin{document}
\title[ On certain generalized notions using $\mathcal{I}$-convergence in topological spaces]
{On certain generalized notions using $\mathcal{I}$-convergence in topological spaces}

\author[  P. Das, U. Samanta, S. Lin]{  Pratulananda Das$^*$, Upasana Samanta$^*$,  Shou Lin$^{\dag}$\ }
\newcommand{\acr}{\newline\indent}
\address{\llap{*\,}Department of Mathematics, Jadavpur University, Kolkata-32, West Bengal, India}
\email{pratulananda@yahoo.co.in, samantaupasana@gmail.com}

\address{\llap{$\dag$\,}Department of Mathematics, Ningde Normal University, Ningde, 352100, Fujian, People’s Republic of China}
\email{shoulin60@163.com}

\thanks{ The first author
is thankful to NBHM for granting the project (sanction no. 02011/9/2022/NBHM(RP)/RD II/10378) during the tenure of which this work was done. }


\maketitle

\begin{abstract}
In this paper, we consider certain topological properties along with certain types of mappings on these spaces defined by the notion of ideal convergence. In order to do that, we primarily follow in the footsteps of the earlier studies of ideal convergence done by using functions (from an infinite set $S$ to $X$) in \cite{CS, das4, das5}, as that is the most general perspective and use functions instead of sequences/nets/double sequences etc. This functional approach automatically provides the most general settings for such studies and consequently extends and unifies the proofs of several old and recent results in the literature about spaces like sequential, Fr\'{e}chet-Uryshon spaces and sequential, quotient and covering maps. In particular, we introduce and investigate the notions of $\ic$-functional spaces, $\ic$-functional continuous, quotient and covering mappings and finally $\ic$-functional Fr\'{e}chet-Uryshon spaces. In doing so, we take help of certain set theoretic and other properties of ideals.
\end{abstract}

\section{\textbf{Introduction}}
\paragraph{}
\vskip .3cm
The idea of statistical convergence of sequence was introduced in \cite{fast,steinhaus} as an extension of the usual notion of convergence. Apart from a lot of investigations in the fields of summability theory, measure theory, functional analysis etc., this idea has led to various investigations in the settings of topological spaces (for example see \cite{lin1,lin2,mk,li,liul,tang,Zhou1,Zhou2}). The most important generalization of almost all types of convergence including statistical convergence had been proposed by Kostyrko et al. \cite{kostyrko2} who had introduced the concepts of $\ic\mbox{-}$convergence and $\ic^{\ast}\mbox{-}$convergence in metric spaces using ideals of the set of all natural
numbers. Following the line of Kostyrko et al., the same has been studied for sequences in general topological spaces \cite{Lahiri1}, for nets in topological and uniform spaces \cite{Lahiri2} (subsequently studied in \cite{das2, das3}) and for functions in topological spaces \cite{CS, das5}, uniform spaces \cite{das4} for example, where other references can be found.\\

On the other hand, it is a well known fact that the topology of a topological space, in general, can not be determined by convergent sequences, unlike metric spaces where sequences play a much more important role in characterizing several notions. From the beginning, it has been a very rich and challenging topic of investigations as to, in which topological spaces sequences play a better role. The first countable, Fr\'{e}chet-Uryshon and sequential spaces are examples of some such spaces that are determined by convergence of sequences \cite{engelking,Lin1}. Instead of usual convergence of sequences, first in \cite{renukadevi,tang} the authors have worked with statistical convergence to define statistical counterparts of Fr\'{e}chet-Uryshon and sequential spaces. Subsequently the more general idea of ideal convergence of sequences has been widely used to introduce these notions as also several other new ideas in topological settings (for example one can see \cite{Zhou1,Zhou2,I-QN,renukadevi, skpal}).

In particular in \cite{Zhou2}, Zhou and his co-authors defined $\ic\mbox{-}$continuous, $\ic\mbox{-}$quotient and $\ic\mbox{-}$covering mappings and checked how they interact with $\ic\mbox{-}$sequential, $\ic\mbox{-}$Fr\'{e}chet spaces.\\

As a natural consequence, in this paper, we further generalize the whole setting of such investigations by considering ideals of an arbitrary infinite set $S$, and as a natural replacement, instead of sequences in $X$ we take functions from $S$ to $X.$ This approach unifies the two directions mentioned above and provides the most general type of results. Primarily we use the idea of $\ic\mbox{-}$convergence of functions to introduce $\ic$-functional open sets, $\ic$-functional closed sets, $\ic\mbox{-}$functional spaces and $\ic$-functional Fr\'{e}chet-Uryshon spaces and establish several properties. We also proceed in the same way to extend the ideas of $\ic\mbox{-}$continuous, $\ic\mbox{-}$quotient and $\ic\mbox{-}$covering mappings and subsequently investigate their counterparts, namely,  $\ic$-functional continuous, quotient and covering mappings and their effects on $\ic$-functional spaces, $\ic$-functional Fr\'{e}chet-Uryshon spaces. In order to clear ambiguity and for the sake of continuity, we call all mappings with domain $S$ ``functions" (continuing the nomenclature of \cite{CS, das4, das5}) and mappings from one topological space to another as just ``mappings".

As a consequence, not only the results of \cite{Zhou1,Zhou2,I-QN,renukadevi, skpal} become special cases of our results, also the whole treatment seems much more simplified, at the same time underscoring the focal point that, several topological concepts can actually be studied without restricting the domain set.\\

\section{\textbf{Preliminaries}}
\vskip .3cm
Let $\bn$ denote the set of all natural numbers and let $K\subset \bn.$ Recall that the \textit{natural} or \textit{asymptotic} density of $K$ is defined by $d(K)=\displaystyle{\lim_{n\longrightarrow \infty}\dfrac{1}{n}|\{k\in \bn:k\leq n\}|}$ if the limit exists. If $X$ is a topological space then a sequence $(x_n:n\in \bn)$ in $X$ is \textit{statistically convergent} to $x\in X$ if for each neighbourhood $U$ of $x$ in $X,~d(\{n\in \bn:x_n\not\in U\})=0$ \cite{mk}.

The notion of statistical convergence has subsequently been extended to the notion of
$\ic\mbox{-}$convergence, which is based on the notion of ideal of subsets of $\bn.$ Let $Y$ be a non-empty set and let $\pc(Y)$ be the family of all subsets of $Y.$  A family $\ic (\subset \pc(Y))$ of subsets of a non-empty set
$Y$ is said to be an \textit{ideal} of $Y$ if $(i)~A,~B\in \ic$
imply $ A\cup B \in \ic ~(ii)~A\in \ic,~B\subset A$ imply
$B\in \ic$, while an \textit{admissible ideal} $\ic$ of $Y$
covers $Y$. Such ideals
are also called free ideals. If $\ic$ is a proper non-trivial ideal
of $Y$ (i.e. $Y\notin \ic,~\ic\neq \emptyset )$, then the
family of sets $\mathcal{F}(\ic)= \{M\subset Y: Y\setminus M
\in \ic\}$ is a \textit{filter} (called the dual filter)  of $Y$
whereas the \textit{coideal} of $\ic$ is $\ic^{+}=\{A\subset Y:A\not\in \ic\}.$
We denote the
ideal consisting of all finite subsets and density zero subsets of
$\bn$ by $\ic_{fin}$ and $\mathcal{I}_{d}$ respectively.

If $\ic$ is a maximal ideal then for any $A\subset S,$ we have
either $A\in \ic$ or $S\setminus A \in \ic.$ For each ideal $\ic$ of $S,$
the set of all maximal ideals $\jc$ of $S$ such that $\ic\subset \jc$
is denoted by $\Theta(\ic).$ It is known that $\ic=\displaystyle{\bigcap_{\jc\in \Theta(\ic)}\jc}$
\cite{givant}. Recall that
$B\subset \bn$ is said to be a \textit{pseudounion} of a family $\ac \subset \pc(\bn)$ if
$\bn\setminus B$ is infinite and $A\setminus B$ is finite for each $A\in \ac$ \cite{I-QN}.\\

A sequence $(x_n:n\in \bn)$ in a topological space $X$ is said to be \textit{$\ic\mbox{-}$convergent}
to $x\in X$ provided for each neighbourhood $U$ of $x,$ the set $\{n\in \bn:x_n\not\in U\}$
belongs to $\ic$ \cite{Lahiri1}. $\ic\mbox{-}$convergence of sequence coincides with
ordinary convergence of sequence if we take $\ic=\ic_{fin}$ and with the
statistical convergence if $\ic=\ic_{d}.$

The concept of $\ic^{*}\mbox{-}$convergence of real sequence arises
from a result of statistical convergence that: a real sequence $(x_n:n\in \bn)$
is statistically convergent to $x$ if and only if there exists a set
$M=\{m_k:k\in \bn\}$ with $m_{1}< m_{2}<\cdots m_{k}\cdots$ such that
$d(M)=1$ and $\displaystyle{\lim_{k\longrightarrow \infty}x_{m_{k}}}=x.$
This idea has been extended to $\ic^{*}\mbox{-}$convergence of a sequence
in a topological space as a sequence $(x_n:n\in \bn)$ in $X$ is \textit{$\ic^{*}\mbox{-}$convergent}
to $x\in X$ if and only if there exists a set $M\in \fc(\ic)$ where
$m_{1}< m_{2}<\cdots< m_{k}<\cdots$ such that $\displaystyle{\lim_{k\longrightarrow \infty}x_{m_{k}}}=x$
\cite{Lahiri1}.

Throughout the paper $X$ stands for a topological space, $S$ an infinite set and
$\ic,$ an admissible ideal of $S$ unless otherwise stated. Further by a ``space" we will always mean a
topological space. Our topological terminology and notation are as in the book
\cite{engelking}.\\

\section{\textbf{$\ic\mbox{-}$functional open sets, $\ic\mbox{-}$functional closed sets and $\ic$\mbox{-}functional space}}

\vskip .3cm

Before we proceed to introduce our main concepts of this  section, we present certain basic observations about convergence of functions which happen to be the main tool behind these generalizations.\\

\begin{Def}
For $x \in X$, we say that a function $f: S \rightarrow X$
\begin{itemize}
\item is \emph{convergent} to $x,$ whenever for every open set $U$ containing $x$, the set $f^{-1}(U)$ is co-finite.
\item is \emph{$\mathcal I$-convergent} to $x,$ whenever for every open set $U$ containing $x$, the set $f^{-1}(U)$ is in $\mathcal{F}(\mathcal{I})$ \cite{CS}.
\item is \emph{$\mathcal{I}^*$-convergent} to $x,$ whenever there is a set $M \in \mathcal{F}(\mathcal{I})$ such that $g$ defined by ``$g(s) = f(s)~ \mbox{if}~ s \in M$ and $g(s) = x$ if $s \notin M$" is convergent to $x$ \cite{CS}.
\end{itemize}
\end{Def}

Suppose $g:S\longrightarrow X,$ is $\ic\mbox{-}$convergent to $x.$ Let $S^{\prime}$ be an infinite subset of $S$ with $|S^{\prime}|=|S|.$ Let $h:S\longrightarrow S^{\prime}$ be a bijective function and let $\Phi=g|_{S^{\prime}}$. Now $\Phi$ is said to be $\ic\mbox{-}$convergent to $x$ if $(\Phi \circ  h)(s)=g(s),~\forall ~s\in S$ is $\ic\mbox{-}$convergent to $x.$

Further if $f:S\longrightarrow X$ is convergent to $x\in X,$ then for any infinite $S^{\prime}\subset S,~f|_{S^{\prime}}$ is convergent to $x.$
In a Hausdorff space $\ic\mbox{-}$limit of a function is unique. For two ideals $\ic\subset \jc$  of $S$, if $f:S\longrightarrow X$ is $\ic\mbox{-}$convergent to $x$ then $f:S\longrightarrow X$ is $\jc\mbox{-}$convergent to $x.$

 Following \cite{I-QN} we can say that an ideal $\ic$ of $S$ has a \textit{pseudounion} if there exists an infinite set $A\subset  S$ with $|S|=|S\setminus A|$ such that $I\setminus A$ is finite for each $I\in \ic.$
\begin{Lemma}
\label{Lm3}
If $\ic$ has a pseudounion and $f:S\longrightarrow X$ is  $\ic\mbox{-}$convergent to $x$ then there exists a function from $S$ to $X$ which is convergent to $x.$
\end{Lemma}
\begin{proof}
Let $f:S\longrightarrow X$ be $\ic\mbox{-}$convergent to $x.$ Since $\ic$ has a pseudounion, there exists an infinite set $A\subset S$ with $S\setminus A\in \ic^{+}$ such that $I\cap (S\setminus A)$ is finite for each $I\in \ic.$ As $f$ is $\ic\mbox{-}$convergent to $x,$ for every open set $O$ containing $x,~A_{O}=\{s\in S:f(s)\not\in O\}\in \ic.$ Thus $A_{O}\cap (S\setminus A)$ is finite. Then $\Phi:S\longrightarrow X$ defined by $\Phi(s)=f(s)$ if $s\in S\setminus A$ and $\Phi(s)=x$ if $s\in A,$  is convergent to $x.$
\end{proof}

\begin{Prop}
\label{Pr1}
Let $g: S \rightarrow X$ be given. Then $g$ is $\ic\mbox{-}$convergent to $x$ if and only if $g$ is $\jc\mbox{-}$convergent to $x$ for each $\jc\in \Theta(\ic).$
\end{Prop}

\begin{Ex}
\label{Ex1}
Let $\ic$ be an ideal of $S.$ Take $\infty \not\in S.$ We define a topology on $S\cup \{\infty\}$ by considering each $s\in S$ isolated and each basic open neighbourhood $U$ of $\infty$ as $(S\setminus I)\cup \{\infty\}$ for some $I\in \ic.$ This space is denoted by $\sum_{S}(\ic).$
Clearly the inclusion mapping $i:S\longrightarrow \sum_{S}(\ic)$ is $\ic\mbox{-}$convergent to $\infty.$ Note that if $\ic\not=\ic_{fin},$ then $\ic$ contains an infinite set $I$ `say'. Then it readily follows that the inclusion function is not convergent to $\infty$ in the usual sense.
\end{Ex}

Let us now look back at the history as to how the notion of closed sets in topological spaces have been generalized using sequences. Recall that a subset $F\subset X$ is called \textit{sequentially closed} if for each sequence $(x_n:n\in \bn)$ in $F$ converging to $x\in X,$ we have $x\in F.$ $X$ is called a \textit{sequential space} \cite{franklin1} if each sequentially closed subset of $X$ is closed. A subset $U\subset X$ is called
\textit{sequentially open} if $X\setminus U$ is sequentially closed. Di Maio and Ko\v{c}inac introduced statistical version of sequential space in \cite{mk} while Pal \cite{skpal} further extended it to $\ic\mbox{-}$sequential spaces. Very recently Zhou et al. revisited the notion of $\ic\mbox{-}$sequential space in \cite{Zhou2} where following notions were introduced. A subset $F\subset X$ is called \textit{$\ic\mbox{-}$closed} if for each sequence $(x_n:n\in \bn)$ in $F,~\ic\mbox{-}$convergent to $x\in X,$ we have $x\in F.$ A subset $U\subset X$ is called \textit{$\ic\mbox{-}$open} if $X\setminus U$ is $\ic\mbox{-}$closed. $X$ is called an \textit{$\ic\mbox{-}$sequential space} if each $\ic\mbox{-}$closed subset of $X$ is closed. Motivated by the generalization of $\ic\mbox{-}$sequential spaces from the idea of sequential spaces,  we now introduce the main concept of this section.
\begin{Def}
\label{D1}
\noindent Let $X$ be a topological space. $(i)$ $F \subset X$ is said to be \textit{$\ic\mbox{-}$functional closed} if for each function $g:S\rightarrow F$ that is $\ic$-convergent to $x\in X$ we have $x\in F.$\\
\noindent $(ii)$ $U \subset X$ is said to be \textit{$\ic\mbox{-}$functional open} if $X\setminus U$ is $\ic$\mbox{-}functional closed.\\
\noindent $(iii)$ $X$ is called an \textit{$\ic\mbox{-}$functional space} if each $\ic\mbox{-}$functional closed subset of $X$ is closed.
\end{Def}
If we consider ``usual" convergence of functions (see Definition 3.1) instead of $\ic\mbox{-}$convergence, we call $\ic\mbox{-}$functional closed sets, $\ic\mbox{-}$functional open sets and $\ic\mbox{-}$functional spaces as functional closed, functional open and functional spaces respectively.
Clearly, every $\ic\mbox{-}$functional closed set is functional closed but the following example shows that the converse is not generally true.
\begin{Ex}
\label{E3}
Let $\ic$ be a maximal ideal of $S.$ We consider the space $\sum_{S}(\ic)$ as in Example \ref{Ex1}. Then $S$ is a functional closed set in $\sum_{S}(\ic)$ but not $\ic\mbox{-}$functional closed.
\end{Ex}
As an immediate consequence of Lemma \ref{Lm3} we can see that
\begin{Prop}
\label{Pr2}
$A\subset X$ is $\ic\mbox{-}$functional closed if and only if $A$ is functional closed provided $\ic$ has a pseudounion. Therefore $X$ is an $\ic\mbox{-}$functional space if and only if $X$ is a functional space provided $\ic$ has a pseudounion.
\end{Prop}

We can modify Definition \ref{D1}$(ii)$ in the following way.
\begin{Lemma}
\label{Lm5}
A subset $O$ of $X$ is $\ic\mbox{-}$functional open if and only if no function $h: S\longrightarrow X\setminus O$ is $\ic\mbox{-}$convergent to a point in $O.$
\end{Lemma}
\begin{proof}
Sufficiency directly follows from Definition \ref{D1}$(i),(ii)$. As $O$ is $\ic\mbox{-}$functional open so $X\setminus O$ is $\ic\mbox{-}$functional closed. Hence for every function $h: S\longrightarrow X\setminus O$ which is $\ic\mbox{-}$convergent to $x,$ we must have $x\in X\setminus O.$
\end{proof}
It is evident that every open set (and so every closed set) is $\ic\mbox{-}$functional open ($\ic\mbox{-}$functional closed). Following example establishes the existence of a space which is not $\ic\mbox{-}$functional.
\begin{Ex}
\label{E4}
Consider the Cartesian product $S\times S.$ For $a\in S,$ we call the subset $S\times \{a\}$ as the $a$-th row of $S\times S.$ Let $\ic$ have a pseudounion and let $\infty$ be an element outside $S\times S.$ Let $X=(S\times S)\cup \{\infty\}.$ We define a topology on $X$ as follows. Let $\tau_{1}=\pc(S\times S)$  and let $\tau_{2}$ be the collection of those subsets $A$ of $X$ so that $\infty\in A$ and $\{a\in S:(\{s\in S:(s,a)\in A\}\in \fc(\ic))\}\in \fc(\ic).$ 
Take $\tau=\tau_{1}\cup \tau_{2}.$ Then it can be verified that $\tau$ is a topology on $X.$

No function from $S$ to $X$ can be $\ic\mbox{-}$convergent to $\infty.$ If $g:S \longrightarrow X$ is $\ic\mbox{-}$convergent to $\infty$ then by Lemma \ref{Lm3} there is a function $f:S\longrightarrow X$ which is convergent to $\infty.$ Note that each row contains at most finitely many elements of the form $f(s).$ Excluding these terms from each row, we obtain an open set containing $\infty$ which contains no terms of the form $f(s).$ Also no function from $S$ to $S\times S$ can be $\ic\mbox{-}$convergent to a point of $S\times S$ unless it is eventually constant. But $\infty$ is a limit point of $S \times S.$ Hence $S\times S$ is $\ic\mbox{-}$functional closed but not closed and therefore $X$ is not $\ic\mbox{-}$functional.
\end{Ex}
However there exists an ideal for which every sequential space is $\ic\mbox{-}$functional.
\begin{Prop}
\label{Pr3}
Let $S=\bigcup_{i\in \bn} S_{i}$ such that $S_i\cap S_{j}=\emptyset$ for different $i,j$ and let $\ic_0 = \{ A\subset S :~ A\cap S_i\not=\emptyset~\mbox{ for finitely many}~ i\}.$  Then every sequential space $X$ is an $\ic_0$-functional space.
\end{Prop}
\begin{proof}
Let $O\subset X$ be $\ic\mbox{-}$functional open. If $O$ is not open then there is a sequence $x_n\in X\setminus O$ converging to $x\in O.$ Define a function $g:S\longrightarrow X\setminus O$ by $g(s)=x_i$ if $s\in S_i.$ Then $g$ is $\ic\mbox{-}$convergent to $x\in O.$ Hence $X\setminus O$ is not $\ic\mbox{-}$functional closed, which is a contradiction.
\end{proof}

\begin{Prop}
\label{Pr4}
The following are equivalent for any $A\subset X.$\\
\noindent $(i)$ $A\subset X$ is $\ic\mbox{-}$functional open\\
\noindent $(ii)$ For any function $g:S\longrightarrow X$ which is $\ic\mbox{-}$convergent to $x\in A,$ we have $\{s\in S:g(s)\in A\}\in \ic^{+}.$ \\
\noindent $(iii)$ $|\{s\in S:g(s)\in A\}|\geq\omega$ for each function $g:S\longrightarrow X$ which is $\ic\mbox{-}$convergent to $x\in A.$
\end{Prop}
\begin{proof}
\noindent $(i)\Longrightarrow (ii)$ Let $A\subset X$ be $\ic\mbox{-}$functional open and let $g:S\longrightarrow X$ be $\ic\mbox{-}$convergent to $x\in A.$ If possible let $C=\{s\in S:g(s)\in A\}\in \ic.$ Fix an element $a\in X\setminus A.$ Define a function $h:S\longrightarrow X\setminus A$ by $h(s)=g(s)$ for $s\in S\setminus C$ and $h(s)=
a$ if $s\in C.$ Let $U$ be a neighbourhood of $x.$ Then $\{s\in S:g(s)\in U\}\cap S\setminus C \subset \{s\in S:h(s)\in U\}\in \fc(\ic).$ Thus $h:S\longrightarrow X\setminus A$ is $\ic\mbox{-}$convergent to $x,$ this contradicts that $A$ is $\ic\mbox{-}$functional open. Thus $(ii)$ holds.\\
\noindent $(ii)\Longrightarrow (iii)$ As $\ic$ is an admissible ideal, thus $(iii)$ holds.\\
\noindent $(iii)\Longrightarrow (i)$ If possible let $A\subset X$ be not $\ic\mbox{-}$functional open. So $X\setminus A$ is not $\ic\mbox{-}$functional closed. Therefore there is a function $g:S\longrightarrow X\setminus A$ which is $\ic\mbox{-}$convergent to $x\in A$ and evidently $\{s\in S:g(s)\in A\} = \emptyset$ which contradicts $(iii).$
\end{proof}


\begin{Ex}
\label{E5}
Let $\ic$ be a maximal ideal of $S$ and let $g:S\longrightarrow X$ be $\ic\mbox{-}$convergent to $x.$ Let $Y=\{g(s):s\in S\}\cup \{x\}.$ Endow $\{g(s):s\in S\} \subset Y$ with the discrete topology and let a basic neighbourhood of $x$ be of the form $\{x\}\cup
\{g(s):s\in A\}$ for some $A\in \fc(\ic).$ $Y$ endowed with this topology is an $\ic\mbox{-}$functional space. To prove that, let $U$ be $\ic\mbox{-}$functional open in $Y.$ Without any loss of generality assume that $x\in U.$ As $\ic$ is maximal, by Proposition \ref{Pr4}(ii), we have $\{s\in S:g(s)\in U\}\in \fc(\ic).$ Hence $\{x\}\cup \{g(s)\in U\}\subset U$ which implies that $U$ is open in $Y.$
\end{Ex}

\begin{Lemma}
\label{Lm6}
Let $X=\displaystyle{\prod_{i\in \Lambda}X_{i}}$ have the product topology. Then a function $f:S\longrightarrow X$ is $\ic\mbox{-}$convergent to $x=(x_i)$ if and only if $\pi_{i}\circ f$ is $\ic\mbox{-}$convergent to $x_i$ for each $i\in \Lambda.$
\end{Lemma}
\begin{proof}
Let $\pi_{i}\circ f$ be $\ic\mbox{-}$convergent to $x_i$ for each $i\in \Lambda.$ Let $O=\displaystyle{\prod_{i\in \Lambda}O_{i}}$ be a basic open set in $X$ containing $x.$ Let $O_i=U_{i}$ for $i=m_{1},m_{2},\cdots,m_{k}$ and $O_i=X_i$ otherwise. Then $\{s\in S:  (\pi_{i}\circ f)(s)\in U_{i}\}\in \fc(\ic) $ for each $i=m_{1},m_{2},\cdots,m_{k}.$ Now $\displaystyle{\bigcap_{i\in\{m_{1},m_{2},\cdots,m_{k}\}} \{s\in S:  (\pi_{i}\circ f)(s)\in U_{i}\}}\in \fc(\ic).$ Consequently the result follows. Clearly the converse holds.
\end{proof}

\begin{Prop}
\label{Pr5}
Let $X=\displaystyle{\prod_{i\in \Lambda}X_i}$ have the product topology and let $O$ be $\ic\mbox{-}$functional open in $X.$ Then $\pi_{i}(O)$ is $\ic\mbox{-}$functional open in $X_i$ for each $i\in \Lambda.$
\end{Prop}
\begin{proof}
If possible let $\pi_{i}(O)$ be not $\ic\mbox{-}$functional open in $X_i.$ Then there exists a function $g:S\longrightarrow X_i$ which is $\ic\mbox{-}$convergent to $x\in \pi_{i}(O)$ and $\{s\in S:g(s)\in \pi_{i}(O)\}\in \ic.$ Now fix some $a_{j}\in \pi_{j}(O)$ for $j\not=i.$ Define a function $h:S\longrightarrow X$ by
\begin{equation*}
(\pi_{j}\circ h)(s) =
\begin{cases}
a_j \, &\text{if} \quad j\not= i\\
g(s) \, &\text{if} \quad j=i. \\
\end{cases}
\end{equation*}
Let $y=(y_i)$ be defined as follows.
\begin{equation*}
y_{j}=
\begin{cases}
a_{j} \, &\text{if} \quad j\not= i,\\
x \, & \text{if} \quad j=i. \\
\end{cases}
\end{equation*}
Then $h: S\longrightarrow X$ is $\ic\mbox{-}$convergent to $y$ (by Lemma \ref{Lm6}). Also $\{s\in S: g(s)\in \pi_{i}(O)\}=\{s\in S:h(s)\in O\}\in \ic,$ contradicts that $O$ is $\ic\mbox{-}$functional open.\\
\end{proof}
We now state certain basic results regarding $\ic\mbox{-}$functional spaces without proofs. \\

(i) Let $\ic\subset \jc$ be two ideals of $S$ and let $X$ be a space. If $U\subset X$ is $\jc\mbox{-}$functional open then it is $\ic\mbox{-}$functional open.

(ii) Let $\ic\subset \jc$ be two ideals of $S.$ If $X$ is $\ic\mbox{-}$functional then it is $\jc\mbox{-}$functional.

(iii) Suppose that $\{\ic_{\alpha}:\alpha\in A\}$ is a collection of ideals of $S.$ If $X$ is a space and $U\subset X$ is $\ic_{\alpha}\mbox{-}$functional open for some $\alpha\in A,$ then $U$ is $\displaystyle{\bigcap_{\alpha \in A}} \ic_{\alpha}\mbox{-}$functional open.\\

\begin{Lemma}
\label{Lm8}
Let $\ic$ be a maximal ideal of $S.$ If $U,~V$ are two $\ic\mbox{-}$functional open subsets of $X$ then $U\cap V$ is also $\ic\mbox{-}$functional open.
\end{Lemma}
\begin{proof}
Let $g:S\longrightarrow X$ be $\ic\mbox{-}$convergent to $x\in U\cap V.$ So, $\{s\in S:g(s)\in U\}\in \fc(\ic)$  and $\{s\in S:g(s)\in V\}\in \fc(\ic)$ (by Proposition \ref{Pr4}). Now $\{s\in S:g(s)\in U\}\cap \{s\in S:g(s)\in V\}=\{s\in S:g(s)\in U\cap V\}\in \fc(\ic)$ and therefore $U\cap V$ is $\ic\mbox{-}$functional open.
\end{proof}

The \textit{$\ic\mbox{-}$functional coreflection} of a space $X$ is the set $X$ endowed with the topology
 generated by $\ic\mbox{-}$functional open subsets of $X$ as a subbase and the topology is denoted by $\ic\mbox{-}fX.$ Clearly for a space $X,~\ic\mbox{-}fX$ is finer than the topology of $X.$ Further If $\ic$ is a maximal ideal of $S,$ then the collection of all $\ic\mbox{-}$functional open sets itself forms a topology on $X.$\\

\begin{Def}
Let $\ic$ be an ideal of $S$ and $A\subset X.$ A function $f:S\longrightarrow X$ is said to be $\ic\mbox{-}$eventually in $A$ if there is a $E\in \ic$ such that $f(s)\in A$ for all $s\in S\setminus E.$
\end{Def}
\begin{Prop}
\label{Pr6}
Let $\ic$ be a maximal ideal of $S$. Then $A\subset X$ is $\ic\mbox{-}$functional open if and only if for each function which is  $\ic\mbox{-}$convergent to a point of $A$, it is $\ic\mbox{-}$eventually in $A.$
\end{Prop}
\begin{proof}
The result follows from Proposition \ref{Pr4}.
\end{proof}
\begin{Th}
\label{TA2}
Every $\ic\mbox{-}$functional space is hereditary with respect to $\ic\mbox{-}$functional open ($\ic\mbox{-}$functional closed) subspaces.
\end{Th}
\begin{proof}
Let $X$ be an $\ic\mbox{-}$functional space. Suppose that $Y$ is an $\ic\mbox{-}$functional open set in $X$. Then $Y$ is open in $X.$ Let $U (\subsetneqq Y)$ be $\ic\mbox{-}$functional open in $Y.$ We have to show that $U$ is $\ic\mbox{-}$functional open in $X.$ Suppose that $g:S\longrightarrow X$ is $\ic\mbox{-}$convergent to $x\in U\subset Y.$ Since $Y$ is open, $\{s\in S:g(s)\in Y\}\in \fc(\ic).$ Let $y\in Y\setminus U.$ Define a function $h:S\longrightarrow Y$ by $h(s)=g(s)$ if $g(s)\in Y$ and $h(s)=y$ if $g(s)\not\in Y.$ Therefore, $h:S\longrightarrow X$ is $\ic\mbox{-}$convergent to $x.$ Since $|\{s\in S:g(s)\not\in U\}|=|\{s\in S:h(s)\not\in U\}|,$  by Proposition \ref{Pr4}, it follows that $U$ is $\ic\mbox{-}$functional open in $X.$ As $X$ is $\ic\mbox{-}$functional space, so $U$ is open in $X$ and so open in  $Y.$

Let $Y$ be an $\ic\mbox{-}$functional closed subset of $X.$ Then $Y$ is closed in $X.$ Let $F (\subsetneqq Y)$ be an $\ic\mbox{-}$functional closed subset of $Y.$ We have to show that $F$ is $\ic\mbox{-}$functional closed subset of $X.$ Suppose that $g:S\longrightarrow F$ is $\ic\mbox{-}$convergent to $x.$ So $x\in Y$ as $Y$ is closed. Therefore $x\in F$ since $F$ is an $\ic\mbox{-}$functional closed subset of $Y.$ Thus $F$ is an $\ic\mbox{-}$functional closed subset of $X,$ so $F$ is a closed subset of $X$ and hence a closed subset of $Y.$
\end{proof}

\begin{Th}
\label{TA3}
$\ic\mbox{-}$functional spaces are preserved by topological sums.
\end{Th}
\begin{Th}
\label{TA4}
Any quotient space of an $\ic\mbox{-}$functional space is an $\ic\mbox{-}$functional space.
\end{Th}
\begin{proof}
Let $X$ be an $\ic\mbox{-}$functional space and let $f:X\longrightarrow Y$ be a quotient mapping. Let  $F\subset Y$ be $\ic\mbox{-}$functional closed. If $F$ is not closed, $f^{-1}(F)$ is not closed (as $f$ is a quotient mapping) and so $f^{-1}(F)$ is not $\ic\mbox{-}$functional closed. Then there exists a function $g:S\longrightarrow f^{-1}(F)$ which is $\ic\mbox{-}$convergent to $x\not\in f^{-1}(F).$ Since $F$ is $\ic\mbox{-}$functional closed and $f$ is continuous, we obtain that $f\circ g:S\longrightarrow F$ is $\ic\mbox{-}$convergent to $f(x)\in F.$ This contradicts that $x\not\in f^{-1}(F).$
\end{proof}
\begin{Th}
\label{TA5}
Every $\ic\mbox{-}$functional space is a quotient of some metric space provided $\ic = \ic_0$, the ideal defined in Proposition \ref{Pr3}.
\end{Th}
\begin{proof}
Let $X$ be an $\ic\mbox{-}$functional space and let $t_n=\dfrac{1}{n+1},~n\in \bn.$ Define a function $f:S\longrightarrow \br$ by $f(s)=t_n$ if $s\in S_n.$ Then $f$ is $\ic\mbox{-}$convergent to $0.$ Take $Y=\{\dfrac{1}{n+1}:n\in \bn\}\cup \{0\}.$ The topology of $Y$ is induced from the usual metric topology of $\br.$ Clearly $O\subset Y$ is open if and only if either $0\not\in O$ or if $0\in O$ then $f(s)\in O$ if $s\in A$ for some $A\in \fc(\ic).$ Let $\cs=\{g:S\longrightarrow X:~g$ is $\ic\mbox{-}$convergent to some $g_{0}\in g(S)\}.$
Writing $\{(g(s):s\in S)\} = Z$, let $d$ be a metric on $Z\times Y = \{(Z, y): y \in Y\}$ defined by $d(Z,a),(Z,b))=|a-b|.$

Now consider the topological sum $$L=\displaystyle{\bigoplus_{Z\in \cs}Z\times Y}.$$ We observe that $A\subset L$ is open if and only if $\{y\in Y:(Z,y)\in A\}$ is open in $Y.$ Consider the mapping $\Phi:L\longrightarrow X$ defined by $\Phi(Z,0)=g_{0}$ and $\Phi(Z,f(s))=g(s).$ Clearly $\Phi$ is onto. Now we show that $\Phi$ is a quotient mapping.

Let $U\subset X$ be open. Then for every $g:S\longrightarrow X,~\ic\mbox{-}$convergent to $a\in U,~\{s\in S:g(s)\in U\}\in \fc(\ic).$ If $(Z,0)\in \Phi^{-1}(U)$ then $g_{0}\in U,$ also $\{s\in S:g(s)\in U\}\in \fc(\ic).$ Write $E=\{s\in S:g(s)\in U\}.$ By the definition of $\Phi,~\Phi(Z,f(s))=g(s)\in U$ for each $s\in E.$ Therefore, $\Phi^{-1}(U)$ is open in $L.$

Again if $U$ is not open in $X$, then  there exists a function $g:S\longrightarrow (X\setminus U)$ which is $\ic\mbox{-}$convergent to $g_{0}\in U.$ Consequently $\{y\in Y:(Z,y)\in \Phi^{-1}(U)\}=\{0\},$ which is not open in $Y.$ Hence $\Phi^{-1}(U)$ is not open in $L.$
\end{proof}

\section{\textbf{$\ic$\mbox{-}functional continuity}}
\vskip .3cm

In this section our main object of investigation is the notion of  $\ic\mbox{-}$functional continuity. Recall that a mapping $f$ from a space $X$ to another space $Y$ is called \textit{sequentially continuous} \cite{boone} provided for any sequentially open set $U$ in $Y,$ $f^{-1}(U)$ is sequentially open in $X.$ It is proved in \cite{boone} that a mapping $f:X\longrightarrow Y$ is sequentially continuous if and only if $f$ \textit{preserves the convergence} of sequences, i.e., for each sequence $(x_n:n\in \bn)$ in $X$ converging to $x,$ the sequence $(f(x_{n}):n\in \bn)$ converges to $f(x).$ In \cite{Zhou2}, authors introduced the notion of $\ic\mbox{-}$continuity in terms of $\ic\mbox{-}$open sets. Extending this notion in the language of functions, we introduce following definitions.
\begin{Def}
\label{D100}
Let $\ic$  be an ideal of $S$ and $f:X\longrightarrow Y$ be a mapping. Then\\
\noindent $(i)$ $f$ is called an \textit{$\ic\mbox{-}$functional convergence preserving mapping} provided for a function $g:S\longrightarrow X,~\ic\mbox{-}$convergent to $x,~f\circ g$ is $\ic\mbox{-}$convergent to $f(x).$\\
\noindent $(ii)$ $f$ is called \textit{$\ic\mbox{-}$functional continuous} provided for any $\ic\mbox{-}$functional open set $U$ in $Y$, $f^{-1}(U)$ is $\ic\mbox{-}$functional open in $X.$
\end{Def}
We call $f$ simply \textit{functional continuous} if we take functional open set instead of $\ic\mbox{-}$functional open set in Definition \ref{D100}.\\
\begin{Lemma}
\label{Lm9}
Let $Y\subset X$ and let $U$ be $\ic\mbox{-}$functional open in $X.$ Then $U\cap Y$ is $\ic\mbox{-}$functional open in $Y.$
\end{Lemma}
\begin{proof}
Let $g:S\longrightarrow Y$ be $\ic\mbox{-}$convergent to $y\in U\cap Y.$ Then $\{s\in S:g(s)\in U\}\in \ic^{+}$ (by Proposition \ref{Pr4}) and therefore $\{s\in S:g(s)\in U\cap Y\}\in \ic^{+}.$
\end{proof}

\begin{Lemma}
\label{LM100}
Let $\ic$ be a maximal ideal of $S$ and let $U\subset Y\subset X.$ Suppose that $U$ is $\ic\mbox{-}$functional open in $Y$ and $Y$ is $\ic\mbox{-}$functional open in $X.$ Then $U$ is $\ic\mbox{-}$functional open in $X.$
\end{Lemma}
\begin{proof}
If $U$ is not $\ic\mbox{-}$functional open in $X$ then there exits a mapping $f:S\longrightarrow X,~\ic\mbox{-}$converging to some $a\in U$ and $f^{-1}(U)\in \ic.$ As $Y$ is $\ic\mbox{-}$functional open in $X, $ and $\ic$ is maximal, $f^{-1}(Y)\in \fc(\ic).$ Therefore $f^{-1}(Y\setminus U) = f^{-1}(Y)\setminus f^{-1}(U) \in \fc(\ic).$ Let $F=f^{-1}(Y\setminus U).$ Define a mapping $\phi:S\longrightarrow Y$ by
\begin{equation*}
\phi(s) =
\begin{cases}
f(s) \, &\text{if} \quad s\in F\\
a \, &\text{if} \quad s\not\in F. \\
\end{cases}
\end{equation*}
Evidently $\phi$ is $\ic\mbox{-}$convergent to $a.$ Also $\phi^{-1}(U)\in \ic,$ contradicts that $U$ is $\ic\mbox{-}$functional open in $Y.$
\end{proof}

\begin{Th}
\label{TA7}
Let $X$ be a space and let $\cu$ be a cover of $X$ by $\ic\mbox{-}$functional open sets. Then a mapping $f:X\longrightarrow Y$ is $\ic\mbox{-}$functional continuous if and only if for each $U\in \cu$ the restriction $f|_{U}$ is $\ic\mbox{-}$functional continuous provided $\ic$ is maximal.
\end{Th}
\begin{proof}
Let $f:X\longrightarrow Y$ be $\ic\mbox{-}$functional continuous and let $U\in \cu.$ Suppose that $V\subset Y$ is $\ic\mbox{-}$functional open. Then $(f|_{U})^{-1}(V)=f^{-1}(V)\cap U$ is $\ic\mbox{-}$functional open in $U.$ Conversely let the condition hold. Then by Lemma \ref{LM100} for any $\ic\mbox{-}$functional open set $V\subset Y,~f^{-1}(V)\cap U$ is $\ic\mbox{-}$functional open in $X.$ As $X=\bigcup \cu,~ f^{-1}(V)=\displaystyle{\bigcup_{U\in \cu}(f|_{U})^{-1}(V)}$ and is $\ic\mbox{-}$functional open as each is so. \\
\end{proof}

In \cite[Theorem 4.2]{Zhou2}, it was shown that every continuous mapping preserves $\ic\mbox{-}$convergence of sequences and if a mapping preserves $\ic\mbox{-}$convergence of sequences then the mapping is $\ic\mbox{-}$continuous. Here also, similar kind of results hold.
\begin{Prop}
\label{Pr7}
Let $X, Y$ be two spaces and $f:X\longrightarrow Y$ be a mapping.
\begin{itemize}
\item [$(i)$] If $f$ is continuous then $f$ preserves $\ic\mbox{-}$functional convergence.
\item [$(ii)$] If $f$ preserves $\ic\mbox{-}$functional convergence then $f$ is $\ic\mbox{-}$functional continuous.
\end{itemize}
\end{Prop}
The examples given below, show that the converses of preceding Proposition are not generally true.
\begin{Ex}
Let $\ic$ be a maximal ideal. Take $X=\sum_{S}(\ic)$ as in Example \ref{Ex1} and let $Y=X$ be endowed with the discrete topology. Let $f:X\longrightarrow Y$ be the identity mapping.

Clearly $i:S\longrightarrow X,$ the inclusion function is not convergent to $\infty.$ Let $S^{\prime}\subset S$ and $i|_{S^{\prime}}.$ If $S^{\prime}\in \ic$ then  $i$ can not converge to $\infty.$

Otherwise take an infinite $S^{\prime\prime}\subset S^{\prime}$ satisfying $|S^{\prime}\setminus S^{\prime\prime}|=|S^{\prime}|.$ If $S^{\prime\prime}\in \ic$ then $(S\setminus S^{\prime\prime}) \cup \{\infty\}$ is
 a neighbourhood of $\infty.$ So $i|_{S^{\prime}}$ again can not be convergent to $\infty.$

 Finally if $S^{\prime\prime}\in \fc(\ic)$ then $S^{\prime\prime}\cup \{\infty\}$ is a neighbourhood of $\infty$ but it does not
contain all but finitely many terms of $i|_{S^{\prime}}.$ So $i|_{S^{\prime}}$ is not convergent to $\infty.$

Therefore there is no convergent function from $S$ to $X$ except for eventual constant mappings. So $f$ preserves $\ic_{fin}\mbox{-}$functional convergence trivially. Thus by Proposition \ref{Pr7}, $f$ is also functional continuous. But evidently $f$ is not continuous.
\end{Ex}

\begin{Ex}
Let $X=\sum_{S}(\ic)$ and let $Y=\{1,0\}$ be endowed with discrete topology. Also let $\ic$ be a non-maximal ideal of $S.$ Then there is $A\subset S$ for which both $A\in \ic^{+}$ and $S\setminus A \in \ic^{+}.$ Define a mapping $g:X\longrightarrow Y$ by $g(x)=1$ if $x\in A$ and $g(x)=0$ otherwise. As $(S\setminus A)\cup \{\infty\}$ is $\ic\mbox{-}$functional open, so $g$ is $\ic\mbox{-}$functional continuous. But $g$ does not preserve $\ic\mbox{-}$functional convergence.
\end{Ex}

Let $f:X\longrightarrow Y$ be an $\ic\mbox{-}$functional continuous mapping and let $g:S\longrightarrow X$ be $\ic\mbox{-}$convergent to $x.$ If $V\subset Y$ is an $\ic\mbox{-}$functional open set containing $f(x)$ then by Proposition \ref{Pr4}, $\{s\in S:g(s)\in f^{-1}(V)\}\in \ic^{+}$ and thus $\{s\in S:(f\circ g)(s) \in V\}\in \ic^{+}.$ This observation leads to the following result immediately.
\begin{Th}
\label{TA8}
Let $\ic$ be a maximal ideal of $S.$ Then a mapping $f:X\longrightarrow Y$ is $\ic\mbox{-}$functional continuous if and only if it preserves $\ic\mbox{-}$functional convergence.
\end{Th}
For the next result we recall the following definition. $\ic$ is called a $P$-ideal if for any $(A_n)_{n \in \omega}$ from $\mathcal{F}(\mathcal{I})$ there is $A \in \mathcal{F}$ such that $A \setminus A_n$ is finite for each n \cite{CS}.
\begin{Th}
\label{TA9}
Let $\ic$ be a $P$-ideal and $X$ be a first-countable space. Then $f:X\longrightarrow Y$ is $\ic\mbox{-}$functional continuous if and only if it preserves $\ic\mbox{-}$functional convergence.
\end{Th}
\begin{proof}
From \cite{CS}, it follows that $\ic\mbox{-}$convergence implies $\ic^{*}\mbox{-}$convergence of functions, as $\ic$ is a $P$-ideal. Let $f:X\longrightarrow Y$ be $\ic\mbox{-}$functional continuous and let $g:S\longrightarrow X$ be $\ic\mbox{-}$convergent to $x.$ Then there is $A\in \ic$ such that $g\downharpoonleft_{S\setminus A}\longrightarrow X$ is convergent to $x.$ Let $U$ be an open neighbourhood of $f(x).$ Then $f^{-1}(U)$ is $\ic\mbox{-}$functional open in $X,$ and so is a functional open set containing $x.$ Consequently $g(s)\in f^{-1}(U)$ for all $s\in S\setminus (A\cup F)$ (for a suitable finite subset $F$ of $S$) and hence $(f\circ g)(s)\in U$ for all $s\in S\setminus (A\cup F).$ By admissibility of $\ic$ it follows that $\{s\in S:(f\circ g)(s)\in U\}\in \fc(\ic).$

The converse result follows directly from Proposition \ref{Pr7}.
\end{proof}
Next we investigate the interrelationships between the notions of continuity and
$\ic\mbox{-}f\mbox{-}$continuity.
\begin{Th}
\label{TA10}
Let $f$ be a mapping from an $\ic\mbox{-}$functional space $X$ to another space $Y.$ Then $f$ is continuous if and only if $f$ is $\ic\mbox{-}$functional continuous.
\end{Th}
\begin{proof}
Let $f:X\longrightarrow Y$ be continuous. Then by Proposition \ref{Pr7}, $f$ is $\ic\mbox{-}$functional continuous.\\
Since every open set is $\ic\mbox{-}$functional open and $X$ is an $\ic\mbox{-}$functional space the converse follows immediately.
\end{proof}

\begin{Cor}
Let $f$ be a mapping from a functional space $X$ to another space $Y.$ Then following are equivalent.\\
\noindent $(1)$ $f$ is continuous.\\
\noindent $(2)$ $f$ preserves $\ic\mbox{-}$functional convergence.\\
\noindent $(3)$ $f$ is $\ic\mbox{-}$functional continuous.\\
\noindent $(4)$ $f$ is functional continuous.
\end{Cor}
\begin{proof}
\noindent $(1)\Longrightarrow(2)$ follows from the Proposition \ref{Pr7}$(i).$ $(2)\Longrightarrow (3)$ follows directly from Proposition \ref{Pr7}$(ii).$ As each functional space is $\ic\mbox{-}$functional space and continuity implies functional continuity, preceding theorem establishes that $(3)\Longrightarrow (4).$ Finally $(4)\Longrightarrow (1)$ since $X$ is a functional space.
\end{proof}

\begin{Th}
Let $X$ be a sequential space and $\ic$ be as defined in Proposition \ref{Pr3}. Then for a mapping $f:X\longrightarrow Y,~f$ is continuous if and only if $f$ preserves $\ic\mbox{-}$functional convergence.
\end{Th}
\begin{proof}
Let $V\subset Y$ be an open set. If $f^{-1}(V)$ is not open then $f^{-1}(V)$ is not $\ic\mbox{-}$functional open (by Proposition \ref{Pr3}). Then there is a function $g:S\longrightarrow X$ which is $\ic\mbox{-}$convergent to $x\in f^{-1}(V)$ such that $\{s\in S:g(s)\in f^{-1}(V)\}\in \ic.$ So $f\circ g:S\longrightarrow Y$ is $\ic\mbox{-}$convergent to $f(x).$ But $\{s\in S:g(s)\in f^{-1}(V)\}\in \ic\Longrightarrow \{s\in S:(f\circ g)(s)\in V\}\in \ic,$ which contradicts that $V$ is $\ic\mbox{-}$functional open. Converse is obvious.
\end{proof}

\begin{Prop}
\label{P100}
If $f:X\longrightarrow Y$ preserves $\jc\mbox{-}$functional convergence for each $\jc\in \Theta(\ic)$ then $f$ preserves $\ic\mbox{-}$functional convergence.
\end{Prop}

\begin{Cor}
If $f:X\longrightarrow Y$ is $\jc\mbox{-}$functional continuous for each $\jc \in \Theta(\ic)$ then $f$ is $\ic\mbox{-}$functional continuous.
\end{Cor}
\begin{proof}
The result follows from Proposition \ref{Pr7} and  \ref{P100}.
\end{proof}

\begin{Ex}
There exists a mapping which preserves $\ic_{fin}\mbox{-}$functional convergence but is not $\jc\mbox{-}$functional continuous for $\jc\in \Theta(\ic_{fin}).$ Let $X=\sum_{S}(\jc)$ and let $Y=S,$ endowed with discrete topology. Define a mapping $f:X\longrightarrow Y$ by $f(s)=s$ if $s\in S$ and $f(\infty)=a$ for some particular $a\in S.$ There is no function from $S$ to $X$ which is convergent. Hence $f$ preserves $\ic_{fin}\mbox{-}$functional convergence but is not $\jc\mbox{-}$functional continuous since $S\setminus \{a\}$ is $\jc\mbox{-}$functional closed and $f^{-1}(S\setminus \{a\})=S$ is not $\jc\mbox{-}$functional closed in $X.$\\
\end{Ex}


\begin{Lemma}
\label{Lm50}
Let $X=\displaystyle{\prod_{i\in \Lambda}X_{i}}$ have the product topology. Then $\pi_{i}:X\longrightarrow X_{i}$ is $\ic\mbox{-}$functional continuous for each $i\in \Lambda.$
\end{Lemma}
\begin{proof}
Follows from Lemma \ref{Lm6}.
\end{proof}
\begin{Prop}
\label{P200}
Let $X=\displaystyle{\prod_{i\in \Lambda}X_{i}}$ have the product topology and let $Y$ be a space. Then a mapping $f:Y\longrightarrow X$ is $\ic\mbox{-}$functional continuous if and only if $\pi_{i}\circ f$ is so for each $i\in \Lambda$ provided $\ic$ is maximal.
\end{Prop}
\begin{proof}
Let $\pi_{i}\circ f$ be $\ic\mbox{-}$functional continuous for each $i\in \Lambda.$ Let $g:S\longrightarrow Y$ be $\ic\mbox{-}$convergent to $y.$ Then $\pi_{i}\circ f\circ g:S\longrightarrow X_i$ is $\ic\mbox{-}$convergent to $(\pi_i\circ f)(y)$ for each $i\in \Lambda$ (by Theorem \ref{TA8}). Using Lemma \ref{Lm6}, it follows that $f\circ g$ is $\ic\mbox{-}$convergent to $f(y).$ Therefore $f$ is $\ic\mbox{-}$functional continuous (by Theorem \ref{TA8}).

Conversely let $f$ be $\ic\mbox{-}$functional continuous. Let $U_{\alpha}\subset X_{i}$ be $\ic\mbox{-}$functional open in $X_{i}$ for some $i\in \Lambda.$ Now $(\pi_{i}\circ f)^{-1}(U_{\alpha})=f^{-1}(\pi_i^{-1}(U_{\alpha}))$ where $\pi_i^{-1}(U_{\alpha})$ is $\ic\mbox{-}$functional open in $X$ (by Lemma \ref{Lm50}). Consequently $f^{-1}(\pi_i^{-1}(U_{\alpha}))$ is $\ic\mbox{-}$functional open in $Y$ as $f$ is $\ic\mbox{-}$functional continuous
\end{proof}

\section{\textbf{$\ic\mbox{-}$functional quotient and $\ic\mbox{-}$functional covering mappings}}
\vskip .3cm

 In the literature (see the papers \cite{boone,Lin1, Lin2, Lin3}), the notions of quotient, sequentially quotient and sequence covering mappings play an important role in studying sequential spaces. These notions have been extended using ideal convergence of sequences to
$\ic\mbox{-}$quotient and $\ic\mbox{-}$covering mappings in \cite{Zhou2}. In this section we intend to further extend these concepts by defining them in terms of functions over an arbitrary set $S.$\\

Let $X,Y$ be two spaces. Recall that an onto mapping $f:X\longrightarrow Y$ is said to be a \textit{quotient} mapping provided $U$ is open in $Y$ if and only if $f^{-1}(U)$ is open in $X;~f$ is said to be \textit{sequentially quotient} \cite{boone} provided $U$ is sequentially open in $Y$ if and only if $f^{-1}(U)$ is sequentially open in $X; ~f$ is said to be \textit{$\ic\mbox{-}$quotient}\cite{Zhou2} provided $U$ is $\ic\mbox{-}$open in $Y$ if $f^{-1}(U)$ is $\ic\mbox{-}$open in $X;~f$ is said to be \textit{sequence covering} \cite{boone} if whenever $(y_n:n\in \bn)$ is a sequence in $Y$ converging to $y\in Y,$ there exists a sequence $(x_n:n\in \bn)$ in $X$ satisfying  $x_{n}\in f^{-1}(y_n)$ for all $n\in \bn$ and $x\in f^{-1}(y)$ such that $x_n$ converges to $x;~f$ is said to be \textit{$\ic\mbox{-}$covering}\cite{Zhou2} if whenever $(y_n:n\in \bn)$ is a sequence in $Y,~\ic\mbox{-}$converging to $y\in Y,$ there exists a sequence $(x_n:n\in \bn)$ in $X$ satisfying $x_{n}\in f^{-1}(y_n)$ for all $n\in \bn$ and $x\in f^{-1}(y)$ such that $(x_n:n\in \bn)$ is $\ic\mbox{-}$convergent to $x.$

Our next definitions are introduced following  this line.
\begin{Def}
\label{D200}
Let $f$ be a mapping from a space $X$ onto another space $Y.$\\
\noindent $(i)$ $f$ is said to be \textit{$\ic\mbox{-}$functional quotient} provided $U$ is $\ic\mbox{-}$functional open in $Y$ if and only if $f^{-1}(U)$ is $\ic\mbox{-}$functional open in $X.$\\
\noindent $(ii)$ $f$ is said to be \textit{$\ic\mbox{-}$functional covering} if for any $g:S\longrightarrow Y$, $\ic\mbox{-}$converging to $y\in Y,$ there exists a function $h:S\longrightarrow X$ satisfying $(f\circ h)(s)=g(s)$ for all $s\in S$ and $x\in f^{-1}(y)$ such that $h$ is $\ic\mbox{-}$convergent to $x.$
\end{Def}
We call $f$ simply \textit{functional quotient} if we take functional open set instead of $\ic\mbox{-}$functional open set in Definition \ref{D200}.
\begin{Th}
\label{TO}
Let $X$ be a space and $Y$ be a non-empty set. Further let $f:X\longrightarrow Y$ be an onto mapping and $\ic$ be a maximal ideal of $S.$ There exists a strongest topology on $Y$ w.r.t which $f$ is $\ic\mbox{-}$functional continuous.
\end{Th}
\begin{proof}
Let $\mathcal{J}=\{V\subset Y:f^{-1}(V)$ is $\ic\mbox{-}$functional open in $X\}.$ Then $\mathcal{J}$ is a topology on $Y$ w.r.t which $f$ is $\ic\mbox{-}$functional continuous. Next let $\mathcal{J}^{\prime}$ be any other topology on $Y$ w.r.t which $f$ is $\ic\mbox{-}$functional continuous. Then for every $\jc^{\prime}\mbox{-}$functional open set $V\subset Y,$ $~f^{-1}(V)$ is $\ic\mbox{-}$functional open in $X.$ So for each $V\in \jc^{\prime},~V\in \jc$ and hence $\jc^{\prime} \subset \jc$ as required.
\end{proof}
In the above Theorem $f$ is an $\ic\mbox{-}$functional quotient mapping.
\begin{Def}
A mapping $f:X\longrightarrow Y$ is said to be \textit{$\ic\mbox{-}$functional open} provided $f(U)$ is $\ic\mbox{-}$functional open in $Y$ whenever $U$ is $\ic\mbox{-}$functional open in $X.$
\end{Def}
\begin{Prop}
Every $\ic\mbox{-}$functional continuous, $\ic\mbox{-}$functional open onto mapping is $\ic\mbox{-}$functional quotient.
\end{Prop}

\begin{Prop}
Let $f:X\longrightarrow Y$ and $g:Y\longrightarrow Z$ be two mappings. Then the following results hold.\\
\noindent $(i)$ If $f$ and $g$ are $\ic\mbox{-}$functional quotient mappings then $g\circ f$ is an $\ic\mbox{-}$functional quotient mapping.\\
\noindent $(ii)$ If $f$ and $g\circ f$ are $\ic\mbox{-}$functional quotient mappings then $g$ is an $\ic\mbox{-}$functional quotient mapping.
\end{Prop}
\begin{proof}
\noindent $(i)$ If $g, f$ are $\ic\mbox{-}$functional continuous then $g\circ f$ is also $\ic\mbox{-}$functional continuous. Again for any $V\subset Z,~(g\circ f)^{-1}(V)=(f^{-1}(g^{-1}(V)))$ and therefore $g\circ f$ is an $\ic\mbox{-}$functional quotient mapping.

$(ii)$ Let $V\subset Z$ such that $g^{-1}(V)$ is $\ic\mbox{-}$functional open in $Y.$ So $(f^{-1}(g^{-1}(V)))$ is $\ic\mbox{-}$functional open in $X$ as $f$ is an $\ic\mbox{-}$functional quotient mapping. Now $(g\circ f)^{-1}(V)=(f^{-1}(g^{-1}(V)))$ and $g\circ f$ being $\ic\mbox{-}$functional quotient, together imply that $V$ is $\ic\mbox{-}$functional open in $Z.$ Next let $V$ be $\ic\mbox{-}$functional open in $Z.$ Then as $g\circ f$ is $\ic\mbox{-}$functional continuous and $f$ is $\ic\mbox{-}$functional quotient, we have $g^{-1}(V)$ is $\ic\mbox{-}$functional open in $Y.$
\end{proof}

\begin{Prop}
$\ic\mbox{-}$functional quotient mappings are preserved by finite products provided $\ic$ is a maximal ideal of $S.$
\end{Prop}
\begin{proof}
Let $f_i:X_{i}\longrightarrow Y_i$ be an $\ic\mbox{-}$functional quotient mapping for $i=1,2,\cdots, N.$ We define a mapping $f:\displaystyle{\prod^{N}_{i=1}}X_{i}\longrightarrow \displaystyle{\prod^{N}_{i=1}}Y_{i}$ by $$f(x_1,x_2,\cdots,x_{N})=(f_{1}(x_1),f_2(x_2),\cdots,f_N(x_N)).$$ By Proposition \ref{P200}, $f$ is $\ic\mbox{-}$functional continuous. It is also onto.

Next let $U\subset \displaystyle{\prod^{N}_{i=1}Y_{i}}$ be such that $f^{-1}(U)$ is $\ic\mbox{-}$functional open. Then there exists a function $g:S\longrightarrow \displaystyle{\prod^{N}_{i=1}}Y_{i},~\ic\mbox{-}$convergent to $y\in U$ and $\{s\in S:g(s)\in U\}\in \ic.$ Consequently $\pi_i\circ g:S\longrightarrow Y_i$ is $\ic\mbox{-}$convergent to $\pi_{i}(y)$ for $i=1,2,\cdots,N$ (by Lemma \ref{Lm6}). If $\{s\in S:(\pi_{i}\circ g)(s)\in \pi_{i}(U)\}\in \fc(\ic)$ for each $i=1,2,\cdots,N$ then $\{s\in S:g(s)\in U\}=\displaystyle{\bigcap^{N}_{i=1}\{s\in S: (\pi_{i}\circ g)(s)\in \pi_{i}(U)\}}\in \fc(\ic),$ which is a contradiction. So there exists $i$ such that $\{s\in S:(\pi_{i}\circ g)(s)\in \pi_{i}(U)\}\not\in \fc(\ic).$ Maximality of $\ic$ implies that $\{s\in S:(\pi_{i}\circ g)(s)\in \pi_{i}(U)\}\in \ic.$ Consequently $\pi_{i}(U)$ is not $\ic\mbox{-}$functional open in $Y_i.$ Thus $f_{i}^{-1}(\pi_{i}(U))$ is not $\ic\mbox{-}$functional open as $f_i$ is $\ic\mbox{-}$functional quotient. But $(\pi_{i}\circ f^{-1})(U)=f^{-1}_{i}(\pi_i(U)),$ which contradicts the fact that $f^{-1}(U)$ is $\ic\mbox{-}$functional open (by Proposition \ref{Pr5}). \\
\end{proof}
The interrelationships results among $\ic\mbox{-}$quotient, quotient and $\ic\mbox{-}$covering mappings that have been studied in \cite{Zhou2} can be further generalized as below.
\begin{Prop}
Let $f$ be a mapping from a space $X$ onto a space $Y.$ \\
\noindent $(i)$ If $f$ is $\ic\mbox{-}$functional continuous and an $\ic\mbox{-}$functional covering mapping then $f$ is an $\ic\mbox{-}$functional quotient mapping.\\
\noindent $(ii)$ If $f$ is one-to-one and an $\ic\mbox{-}$functional quotient mapping then $f$ is an $\ic\mbox{-}$functional covering provided $\ic$ is a  maximal ideal.

\end{Prop}
\begin{proof}

\noindent $(1)$ Let $f:X\longrightarrow Y$ be an $\ic\mbox{-}$functional continuous and $\ic\mbox{-}$functional covering mapping. Suppose $U\subset Y$ is such that $f^{-1}(U)$ is $\ic\mbox{-}$functional open. If $U$ is not $\ic\mbox{-}$functional open there exists a function $g:S\longrightarrow Y,~\ic\mbox{-}$converging to $y\in U$ for which $\{s\in S:g(s)\in U\}\in \ic.$ As $f$ is $\ic\mbox{-}$functional covering there is a function $h:S\longrightarrow X,~\ic\mbox{-}$convergent to $x\in X$ such that $f\circ h=g$ and $f(x)=y.$ Also $\{s\in S:g(s)\in U\}=\{s\in S:(f\circ h)(s)\in U\}=\{s\in S:h(s)\in f^{-1}(U)\}$ which implies that $f^{-1}(U)$ is not $\ic\mbox{-}$functional open in $X.$\\

\noindent $(ii)$ Let $f:X\longrightarrow Y$ be an one-to-one and $\ic\mbox{-}$functional quotient mapping. Let $g:S\longrightarrow Y$ be $\ic\mbox{-}$convergent to $y.$ As $f$ is one-to-one and onto, for each $s\in S$ there exists an unique $x_{s}\in X$ such that $f(x_s)=g(s).$ Define a function $h:S\longrightarrow X$ by $h(s)=x_{s}$ and let $f(x)=y.$ If $h$ is not $\ic\mbox{-}$convergent to $x,$ there exists an open set $O$
containing $x$ such that $\{s\in S:h(s)\in O\}\not\in \fc(\ic).$ Since $\ic$ is maximal, $\{s\in S:h(s)\in O\}\in \ic,$ so $\{s\in S:(f\circ h)(s)\in f(O)\}\in \ic$ (as $f$ is one-to-one) i.e. $\{s\in S:g(s)\in f(O)\}\in \ic$ (because $f\circ h=g$). Now $f(O)$ is $\ic\mbox{-}$functional open in $Y$ as $f^{-1}(f(O))=O$ is $\ic\mbox{-}$functional open in $X,$ and $f$ is $\ic\mbox{-}$functional quotient. This contradicts Proposition \ref{Pr4}.
\end{proof}

The next result establishes when an $\ic\mbox{-}$functional quotient mapping becomes a quotient mapping and conversely.
\begin{Th}
\label{T503}
Let $f$ be a continuous mapping from an $\ic\mbox{-}$functional space $X$ to another space $Y.$ Then $f$ is a quotient map if and only if $f$ is $\ic\mbox{-}$functional quotient and $Y$ is an $\ic\mbox{-}$functional space.
\end{Th}
\begin{proof}
First let $f:X\longrightarrow Y$ be a quotient mapping and let $F\subset Y$ be not closed. Then $f^{-1}(F)$ is not closed in $X$ as $f$ is a quotient mapping. So $f^{-1}(F)$ is not $\ic\mbox{-}$functional closed since $X$ is an $\ic\mbox{-}$functional space. Hence by Theorem 3.4 there exists a function $g:S\longrightarrow f^{-1}(F)$ which is $\ic\mbox{-}$convergent to $x\in X\setminus f^{-1}(F).$ Therefore, $f\circ g:S\longrightarrow F$ is $\ic\mbox{-}$convergent to $f(x)\in Y\setminus F$ and consequently $F$ is not $\ic\mbox{-}$functional closed. Thus $Y$ is an $\ic\mbox{-}$functional space. Let $F\subset Y$ be such that
$f^{-1}(F)$ is $\ic\mbox{-}$functional closed. Now if $F$ is not $\ic\mbox{-}$functional closed, $F$ is not closed. Thus $f^{-1}(F)$ is not closed (as $f$ is a quotient mapping). Since $X$ is an $\ic\mbox{-}$functional space, $f^{-1}(F)$ is not $\ic\mbox{-}$functional closed.

Conversely let $f$ be an $\ic\mbox{-}$functional quotient mapping and let $Y$ be an $\ic\mbox{-}$functional space. Take $  F\subset Y$ such that $f^{-1}(F)$ is closed in $X$ and so $\ic\mbox{-}$functional closed. Since $f$ is an $\ic\mbox{-}$functional quotient mapping $F$ is $\ic\mbox{-}$functional closed. $Y$ is an $\ic\mbox{-}$functional space, thus $F$ is closed. This concludes that $f$ is a quotient mapping.
\end{proof}

Generally an $\ic\mbox{-}$functional quotient mapping is not quotient and vice-versa.
The next result characterises $\ic\mbox{-}$functional spaces in terms of the interrelations of $\ic\mbox{-}$functional quotient and quotient mappings, which can be proved in the same way as that of \cite[Theorem 5.6]{Zhou2}.
\begin{Th}
Let $X$ be a space and $\ic$ be a maximal ideal of $S.$ Then $X$ is an $\ic\mbox{-}$functional space if and only if each $\ic\mbox{-}$functional quotient mapping onto $X$ is quotient.
\end{Th}

\begin{Th}
Let $\ic$ be a maximal ideal. An onto mapping $p:X\longrightarrow Y$ is $\ic\mbox{-}$functional quotient if and only if it has the property that for any space $W$ and a mapping $f:Y\longrightarrow W,~\ic\mbox{-}$functional continuity of $f\circ p$ implies that of $f.$
\end{Th}
\begin{proof}
First let $p:X\longrightarrow Y$ be $\ic\mbox{-}$functional quotient and let $f$ be a mapping from $Y$ to some space $W.$ Take $f\circ p,$ an $\ic\mbox{-}$functional continuous mapping.
Let $F\subset W$ be $\ic\mbox{-}$functional closed. Then $(f\circ p)^{-1}(F)=p^{-1}(f^{-1}(F))$ is $\ic\mbox{-}$functional closed in $X.$ Since $p:X\longrightarrow Y$ is $\ic\mbox{-}$functional quotient, $f^{-1}(F)$ is $\ic\mbox{-}$functional closed in $X.$

Conversely let the condition hold. Consider $F\subset Y$ such that $p^{-1}(F)$ is $\ic\mbox{-}$functional closed. Let $W=\{0,1\}.$ Define a mapping $f:Y\longrightarrow W$ by $f(y)=1$ if $y\in F$ and $f(y)=0$ if $y\in Y\setminus F.$ So $(f\circ p)(x)=1$ if $x\in p^{-1}(F)$ and $(f\circ p)(x)=0$ if $x\not\in p^{-1}(F).$ As in Theorem \ref{TO}, the topology on $W$ is induced by $f\circ p.$ As $(f\circ p)^{-1}(\{1\})=p^{-1}(F)$ is $\ic\mbox{-}$functional closed in $X,~\{1\}$ is $\ic\mbox{-}$functional closed in $W.$ Therefore by $\ic\mbox{-}$functional continuity of $f$, $F$ is $\ic\mbox{-}$functional closed in $Y.$ Consequently $p:X\longrightarrow Y$ is $\ic\mbox{-}$functional quotient.\\
\end{proof}

If $\ic$ has a pseudounion, then the next theorem readily follows from Proposition \ref{Pr2}.
\begin{Th}
\label{T090}
If $f$ is a mapping from a space $X$ onto $Y$ and if $\ic$ has a pseudounion then following results hold.\\
\noindent $(1)$ $f$ is $\ic\mbox{-}$functional continuous if and only if $f$ is functional continuous.\\
\noindent $(2)$ $f$ is $\ic\mbox{-}$functional quotient if and only if $f$ is functional quotient.
\end{Th}

\begin{Th}
Let $f:X\longrightarrow Y$ be a mapping and let $\ic$ have a pseudounion. Then $f$ is $\ic\mbox{-}$functional quotient
if and only if for any function $g:S\longrightarrow Y,~\ic\mbox{-}$converging to $p$ there exists a function $h:S\longrightarrow X$ which is $\ic\mbox{-}$convergent to some $x\in X$ so that $(f\circ h)(S)\subset g(S)$ and $f(x)=p.$
\end{Th}
\begin{proof}
Let $f$ be a $\ic\mbox{-}$functional quotient mapping. Then $f$ is functional quotient (by Theorem \ref{T090}). Let $g:S\longrightarrow Y$ be $\ic\mbox{-}$convergent to $p.$ Since $\ic$ has pseudounion, as in Lemma \ref{Lm3}, we obtain a function $g^{\prime}:S\longrightarrow Y$ which is convergent to $p.$ Now $f^{-1}(g^{\prime}(S)\setminus \{p\})$ is not functional closed, and so we get a function $k:S\longrightarrow f^{-1}(g^{\prime}(S)\setminus \{p\}),$ converging to $a\not\in f^{-1}(g^{\prime}(S)\setminus \{p\}).$ Now $(f\circ k):S\longrightarrow g^{\prime}(S)$ is convergent to $f(a)=p.$

Conversely let the condition hold and let $F\subset Y$ so that $f^{-1}(F)$ is an $\ic\mbox{-}$functional closed set in $X.$ Let $g:S\longrightarrow F$ be $\ic\mbox{-}$convergent to $y\in Y.$ Then there exists an infinite set $S^{\prime}\subset S$ and $g|_{S^{\prime}}$ is convergent to $y.$ Let $\Phi=g|_{S^{\prime}}$ and $h:S\longrightarrow S^{\prime}$ be an onto mapping satisfying $(\Phi\circ h)(s)=g(s)$ for every $s\in S.$ Then $(\Phi \circ h)$ is convergent to $y.$ So, there is a function $k:S\longrightarrow f^{-1}((\Phi \circ h)(S))$ that is $\ic\mbox{-}$convergent to $x\in f^{-1}(y).$ As $f^{-1}(F)$ is $\ic\mbox{-}$functional closed, $x\in f^{-1}(F)$ so $y\in F.$
\end{proof}

\section{\textbf{$\ic\mbox{-}$functional Fr\'{e}chet-Uryshon space}}
\vskip .3cm

Recall that a space $X$ is \textit{Fr\'{e}chet-Uryshon}\cite{franklin1} (resp., \textit{statistically Fr\'{e}chet-Uryshon}\cite{mk}, \textit{$\ic\mbox{-}$Fr\'{e}chet-Uryshon} \cite{Zhou2}) if for each $A\subset X$ with $x\in cl(A)$ there exists a sequence in $A$ which is convergent (resp. statistically convergent, $\ic\mbox{-}$convergent) to $x.$
We can extend these notions to  \textit{$\ic\mbox{-}$functional Fr\'{e}chet-Uryshon space} in the following way.
\begin{Def}
A space $X$ is called an \textit{$\ic\mbox{-}$functional Fr\'{e}chet-Uryshon space} if for each $A\subset X$ and each $x\in cl(A)$ there exists a function $f:S\longrightarrow A, ~\ic\mbox{-}$convergent to $x.$
\end{Def}
\begin{Prop}
\label{Pr000}
Every $\ic\mbox{-}$functional Fr\'{e}chet-Uryshon space is an $\ic\mbox{-}$functional space.
\end{Prop}

We modify \cite[Example 3.1]{tang} to show that the converse of the above Proposition does not hold for a maximal ideal.
\begin{Ex}
Let $\ic$ be a maximal ideal of $S$ and let $X$ be a non-empty set. For every $a\in S,$ let $g_{a}$ be a function from $S$ to $X$ and take $G_{a}=g_{a}(S).$ For fixed $a\in S$ assume that $g_{a}(x)\not=g_{a}(y)$ for $x\not=y\in S.$ Take $G=\{x_{a}:a\in S\}$ with $x_{a}\not=x_{b}$ for $a\not=b\in S.$ Consider a point $\infty$ outside $\bigcup\{G_{a}:a\in S\}\bigcup G$ and take $X=\bigcup\{G_{a}:a\in S\}\bigcup G\bigcup\{\infty\}.$

The topology on $X$ is defined in the following way:\\
\noindent $(1)$ Each point $g_{a}(s)$ is isolated;\\
\noindent $(2)$ For each $a\in S,$ an open neighbourhood of $x_{a}$ is taken as a set of the form $\{x_{a}\}\cup M_{a}$ where $M_{a}=\{g_{a}(s):s\in F\}$ for some $F\in \fc(\ic)$;\\
\noindent $(3)$ Each open neighbourhood of $\infty$ is a set of the form $\{\infty\}\cup M\cup \{M_{a}:a\in M\}$ where $M=\{x_{a}:a\in F\}$ for some $F\in \fc(\ic).$

First to show that $X$ is an $\ic\mbox{-}$functional space, take $Y\subset X$ which is $\ic\mbox{-}$functional closed. Also let $y\in cl(Y).$ Now one of the three cases may occur.\\

\textbf{Case 1:} If $y\in \bigcup\{G_{a}:a\in S\}$ then $\{y\}$ is an open neighbourhood of $y$ and $y\in cl(Y)$ implies that $y\in Y.$

\textbf{Case 2:} If $y\in G,$ then $y=x_{a}$ for some $a\in S.$ If possible let $y\not\in Y.$ Now if $\{s\in S:g_{a}(s)\in Y\}\in \ic$ then $U=(G_{a}\setminus Y)\cup \{x_{a}\}$ is an open neighbourhood of $x_{a}.$ Also $U\cap Y=\emptyset$ contradicts that $y\in cl(Y).$ We define a function $h:S\longrightarrow Y\cap G_{a}$ by $h(s)=g_{a}(s)$ if $g_{a}(s)\in Y$ and $h(s)=g_{a}(s_{0})$ for some fixed $s_{0}\in S.$ Then $h$ is $\ic\mbox{-}$convergent to $x_{a}.$ Because $Y$ is $\ic\mbox{-}$functional closed,
$x_{a}=y\in Y,$ which is a contradiction. Therefore $y\in Y.$

\textbf{Case 3:} Consider the case when $y=\infty.$ If $A=\{s\in S:x_{s}\in Y\}\in \ic$ put $V=X\setminus \bigcup\{G_{a}:a\in A\}.$ Then $V$ is an open neighbourhood of $\infty$ and $V\cap Y=\emptyset,$ which contradicts  $\infty \in cl(Y).$ Hence as in Case 2, we obtain a function $f:S\longrightarrow Y,~\ic\mbox{-}$convergent to $\infty.$ As $Y$ is $\ic\mbox{-}$functional closed, $\infty\in Y.$\\

To prove that $X$ is not an $\ic\mbox{-}$functional Fr\'{e}chet-Uryshon space, take $\infty\in cl(X\setminus (\{x_{a}:a\in S\}\cup \{\infty\})).$ Let $h:S\longrightarrow X\setminus (\{x_a:a\in S\}\cup \{\infty\})$ be $\ic\mbox{-}$convergent to $\infty.$ Define $A_{a}=h(S)\cap G_{a},~a\in S.$ If for each $a\in S,~\{s\in S:g_{a}(s)\in h(S)\}\in \ic$ then $V=\{\infty\}\cup \{G_{a}\setminus A_{a}:a\in S\}$ is an open set containing $\infty.$ But $V\cap h(S)=\emptyset$ and this contradicts that $h$ is $\ic\mbox{-}$convergent to $\infty.$ So there is $a\in S$ such that $\{s\in S:g_{a}(s)\in h(S)\}\in \fc(\ic).$ As in case 2, a function $k:S\longrightarrow h(S)\cap G_a$ can be constructed which is $\ic\mbox{-}$convergent to $x_{a}.$

Next if $\{s\in S:h(s)\in G_{a}\}\in \ic$ then $k:S\longrightarrow h(S)\cap G_{a}$ can be made to $\ic\mbox{-}$converge to a point different from $x_{a}.$ This contradicts Lemma \ref{Lm2}. Now if $\{s\in S:h(s)\in G_{a}\}\in \fc(\ic)$ then $k$ can be made to be $\ic\mbox{-}$convergent to $\infty\not=x_{a},$ which again is a contradiction as $X$ is a Hausdorff space.\\
\end{Ex}

Considering functions over $S$ instead of functions over $\bn,$ we restate \cite[Theorem 6.3]{Zhou2}
as follows.
\begin{Th}
\label{T000}
A space $X$ is an $\ic\mbox{-}$functional Fr\'{e}chet-Uryshon space if and only if each subset of $X$ is an $\ic\mbox{-}$functional space.
\end{Th}

\begin{Prop}
\label{Pr0000}
Subspaces of an $\ic\mbox{-}$functional Fr\'{e}chet-Uryshon space are $\ic\mbox{-}$functional Fr\'{e}chet-Uryshon space.
\end{Prop}

\begin{Th}
\label{T2020}
$\ic\mbox{-}$functional Fr\'{e}chet-Uryshon spaces are preserved by topological sums.
\end{Th}
\begin{proof}
Let $\{X_{a}:a\in \Lambda\}$ be a disjoint family of $\ic\mbox{-}$functional Fr\'{e}chet-Uryshon spaces and let $X=\displaystyle{\bigoplus_{a\in \Lambda}X_{a}}$ be its topological sum. From Proposition \ref{Pr000} and Theorem \ref{TA3}, $X$ is an $\ic\mbox{-}$functional space. For every $Y\subset X$ and for every $a\in \Lambda,~Y\cap X_{a}$ is an $\ic\mbox{-}$functional Fr\'{e}chet-Uryshon space in $X_{a},$ and  consequently is an $\ic\mbox{-}$functional space in $X_{a}.$ Hence the topological sum $\displaystyle{\bigoplus_{a\in \Lambda}Y\cap X_{a}}$ becomes an $\ic\mbox{-}$functional space. As $Y$ is an $\ic\mbox{-}$functional space, therefore by Theorem \ref{T000}, $X$ is an $\ic\mbox{-}$functional Fr\'{e}chet-Uryshon space.
\end{proof}
Although the line of proof of the next result is analogous with that of \cite[Lemma 6.9]{Zhou2}, it has its own significance.
\begin{Th}
\label{T701}
Every space is a continuous and $\ic\mbox{-}$functional covering image of an $\ic\mbox{-}$functional Fr\'{e}chet-Uryshon space provided $\ic$ is a maximal ideal of $S.$
\end{Th}

In general product of two $\ic\mbox{-}$functional Fr\'{e}chet-Uryshon spaces may not be an $\ic\mbox{-}$functional Fr\'{e}chet-Uryshon space. This follows from a modification of an Example from \cite{renukadevi}. Actually product of two $\ic\mbox{-}$functional Fr\'{e}chet-Uryshon spaces need not be an $\ic\mbox{-}$functional space either.\\

A mapping $f:X\longrightarrow Y$ is called \textit{pseudo-open} \cite{Av} if for each $y\in Y$ and an open subset $U$ in $X$ with $f^{-1}(y)\subset U,~f(U)$ is a neighbourhood of $y$ in $Y.$
\begin{Th}
Let $f$ be a pseudo-open mapping from an $\ic\mbox{-}$functional Fr\'{e}chet-Uryshon space $X$ onto a space $Y.$ If $f$ preserves $\ic\mbox{-}$functional convergence then $Y$ is an $\ic\mbox{-}$functional Fr\'{e}chet-Uryshon space.
\end{Th}
\begin{proof}
Let $f$ preserve $\ic\mbox{-}$functional convergence. Let $A\subset Y$ and choose $y\in cl(A).$
If $f^{-1}(\{y\})\cap (cl(f^{-1}(A)))=\emptyset$ then $f^{-1}(y)\in X\setminus (cl(f^{-1}(A))).$ Since $f$
is pseudo-open, $y\in int(f(X\setminus cl(f^{-1}(A)))=int(f(int(X\setminus f^{-1}(A))))\subset int(f(X\setminus f^{-1}(A)))=int(Y\setminus A)=Y\setminus cl(A).$ So, $y\in Y\setminus cl(A),$ a contradiction. Therefore, there exists $x\in f^{-1}(\{y\})\cap (cl(f^{-1}(A))).$ Since $X$ is an $\ic\mbox{-}$functional Fr\'{e}chet-Uryshon space, there is a function $g:S\longrightarrow f^{-1}(A),~\ic\mbox{-}$convergent to $x.$ Hence $f\circ g:S\longrightarrow A$ is $\ic\mbox{-}$convergent to $f(x)=y,$ which consequently implies that $Y$ is an $\ic\mbox{-}$functional Fr\'{e}chet-Uryshon space.
\end{proof}
\begin{Cor}
\label{C1001}
$\ic\mbox{-}$functional Fr\'{e}chet-Uryshon spaces are preserved by continuous pseudo-open mappings.
\end{Cor}

\begin{Th}
A space $Y$ is an $\ic\mbox{-}$functional Fr\'{e}chet-Uryshon space if every mapping onto $Y$ that preserves $\ic\mbox{-}$functional convergence is pseudo-open provided $\ic$ is a maximal ideal of $S.$
\end{Th}
\begin{proof}
Let $Y$ be a space and let $g:S\longrightarrow Y$ be $\ic\mbox{-}$convergent to $y_{g}.$ Consider
 $S_{g}=g(S)\cup \{y_{g}\}$ and $\cs$ is the family of all functions on $S$ which are $\ic\mbox{-}$convergent in $Y.$ A topology on $S_{g}$ is defined as in Example \ref{E5} and is denoted by $S^{\ic}_{g}.$
Clearly $g$ is $\ic\mbox{-}$convergent to $y_{g}$ in $S^{\ic}_{g}.$ In view of Example \ref{E5}, Theorem \ref{TA2} and Theorem \ref{T000}, one can conclude  that $S^{\ic}_{g}$ is an
$\ic\mbox{-}$functional Fr\'{e}chet-Uryshon space since every subset of $S_{g}$ is open or closed in $S^{\ic}_{g}.$
Let $Z=\displaystyle{\bigoplus_{g\in \cs}S^{\ic}_{g}}$ be the topological sum of $\{S^{\ic}_{g}\}.$
By Theorem \ref{T2020}, $Z$ is an $\ic\mbox{-}$functional Fr\'{e}chet-Uryshon space. We define a mapping $f:Z\longrightarrow Y$ such that $f|_{S_{g}}:S^{\ic}_{g}\longrightarrow (S_{g},\tau_{S_{g}})$ is the identity mapping. Then $f$ preserves $\ic\mbox{-}$functional convergence and therefore by Theorem 6.4 $Y$ is an $\ic\mbox{-}$functional Fr\'{e}chet-Uryshon space.
\end{proof}
But suitably modifying \cite[Theorem 6.7]{Zhou2}, we can obtain the next theorem.
\begin{Th}
\label{T700}
Let $Y$ be an $\ic\mbox{-}$functional Fr\'{e}chet-Uryshon space. Then every $\ic\mbox{-}$functional covering mapping from a space onto $Y$ is pseudo-open.
\end{Th}
We end the section with the following interesting observation.
\begin{Th}
A space $X$ is an $\ic\mbox{-}$functional Fr\'{e}chet-Uryshon space if and only if every continuous $\ic\mbox{-}$functional covering mapping onto $X$ is pseudo-open provided $\ic$ is a maximal ideal of $S.$
\end{Th}
\begin{proof}
This follows from Theorem \ref{T700}, Theorem \ref{T701} and corollary \ref{C1001}.
\end{proof}


\end{document}